\documentclass[10pt]{article}

\usepackage[a4paper, total={6in, 9in}]{geometry}
\usepackage{color}
\usepackage{graphicx}
\usepackage{latexsym}
\usepackage{amsfonts,amsmath,amssymb,amsthm}

\newtheorem{theorem}{Theorem}

\newtheorem{lemma}{Lemma}
\newtheorem{conjec}{Conjecture}

\newcommand{\stf}[2]{\genfrac[]{0pt}{}{#1}{#2}}
\newcommand{\sts}[2]{\genfrac\{\}{0pt}{}{#1}{#2}}

\begin{document}

\author{Andr\'as Bazs\'o\\
Institute of Mathematics\\
University of Debrecen\\
P. O. Box 400, H-4002 Debrecen, Hungary\\
and\\
HUN-REN-UD Equations, Functions, Curves and their Applications Research Group\\
E-mail: bazsoa@science.unideb.hu
\and
Istv\'an Mez\H{o}\\ School of Mathematics and Statistics,\\ Nanjing University of Information Science and Technology, \\ Nanjing, Jiangsu, P. R. China \\and\\Institute of Applied Mathematics, John von Neumann Faculty of Informatics, \\ \'Obuda University, Budapest, Hungary\\  E-mail: istvanmezo81@gmail.com
\and
{\'A}kos Pint{\'e}r\\
Institute of Mathematics\\
University of Debrecen\\
P. O. Box 400, H-4002 Debrecen, Hungary\\
E-mail: apinter@science.unideb.hu
\and
Szabolcs Tengely\\
Institute of Mathematics\\
University of Debrecen\\
P. O. Box 400, H-4002 Debrecen, Hungary\\
E-mail: tengely@science.unideb.hu}

\title{Singmaster-type results for Stirling numbers and some related diophantine equations}
\maketitle

\begin{abstract}Motivated by the work of David Singmaster, we study the number of times an integer can appear among the Stirling numbers of both kinds. We provide an upper bound for the occurrences of all the positive integers, and present certain questions for further study. Some numerical results and conjectures concerning the related diohantine equations are collected.
\end{abstract}

\section{Introduction}

The number of times a given positive integer $a$ occurs in the Pascal triangle is usually denoted by $N(a)$. D. Singmaster showed in 1971 \cite{Singm1}, that
\[N(a)=O(\log a).\]
He conjectured that $N(a)$ is, in fact, bounded:
\[N(a)=O(1).\]
This is \emph{Singmaster's conjecture}. A recent result shows that this conjecture indeed holds true over a triangular area inside the whole triangle \cite{Taoetal}.

It is known, that $N(3003)=8$ (see again \cite{Taoetal} for the numerical values). Furthermore, $N(a)=6$ for infinitely many $a$, see \cite{AMM}, \cite{Singm2}, \cite{deweger} and the references therein.

Bounds that are valid over all the triangle are found in the two papers of Singmaster \cite{Singm1,Singm2}, and in \cite{AEH,Kane1,Kane2}. The best result, up to now, is \cite{Kane2}
\[N(a)=O\left(\frac{(\log a)(\log\log\log a)}{(\log\log a)^3}\right).\]
The same question was recently studied with respect to the multinomial coefficients by Koninck et al. \cite{Konincketal}. For a diophantine viewpoint of this problem we refer to \cite{deweger}.
In \cite{FPP} the authors proved that all the solutions to the equation concerning the equal values of Stirling numbers of the second kind
$$
\sts{x}{k}=\sts{y}{l}
$$
in integers $x,k,y,l$ with $\max\{k,l\}\leq 50$ are $(5,2,6,5)$ and $(13,2,91,90)$. Their proof is based on the well-known Baker-Davenport reduction method and they conjectured that there is no other possible solution. Since 
\[\sts{14}{11}=\sts{364}{363}=66\,066,\]
(see Section 4) this conjecture is not true.

In this note we study the same set of problems with respect to the Stirling numbers: let $M_i(a)$ be the number of times the non-negative integer $a$ appears among the Stirling numbers of the $i$-th kind ($i=1,2$). Among other things, we are going to set explicit bounds on $M_i(a)$ by using the Lambert $W$ function. Our proof is based on the idea of Singmaster in \cite{Singm1}, but certain estimations which were elementary in the binomial coefficient case need more careful analysis here.

\section{The estimations with respect to $M_i(a)$}

\subsection{Bound for $M_2(a)$}

In order to derive our bounds on $M_2(a)$, we first need a lemma.

\begin{lemma}1) The sequence $\sts{i+j}{i}$ is monotone increasing both in $i$ and in $j$.

2) The sequence $\sts{2n}{n}$ of the central Stirling numbers of the second kind is strictly increasing.
\end{lemma}

\begin{proof}
The proofs are straightforwardly coming from the recurrence
\[\sts{n}{k}=\sts{n-1}{k-1}+k\sts{n-1}{k}.\]
For example,
\[\sts{i+j}{i}=\sts{i+j-1}{i-1}+i\sts{i+j-1}{i}\ge\sts{i-1+j}{i-1},\]
so $\sts{i+j}{i}$ is indeed monotone increasing in $i$. The other statements are similar.
\end{proof}

To express the statement with respect to $M_2(a)$, we need the Lambert $W$ function which can be defined implicitly as the solution of the transcendental equation \cite{W,Mezo}
\[W(a)e^{W(a)}=a.\]
It is known, that
\begin{equation}
W(a)\sim\log a-\log\log a\quad(a\to\infty).\label{Wasymp}
\end{equation}

We can now prove the following statement.

\begin{theorem}For all positive integer $a\ge2$
\[M_2(a)\le2+2\frac{\log a}{W\left(\frac12\log a\right)}.\]
If $a\to\infty$,
\[M_2(a)=O\left(\frac{\log a}{\log\log a-\log\log\log a}\right).\]
\end{theorem}

\begin{proof}Let us fix an $a>1$. By the second point of the above lemma, there exists a unique smallest integer $b$ such that
\[\sts{2b}{b}>a\]
still holds. If we attempt to solve the equation
\begin{equation}
\sts{i+j}{j}=a,\label{st2eq}
\end{equation}
we see by the first point of the lemma that it can have at most one solution in $i$ for all fixed $j$ and vice versa.

Because of the definition of $b$, it cannot hold true that both $i$ and $j$ are greater than equal to $b$. That is, we must have that for our potential solutions $i<b$ or $j<b$. It therefore follows that \eqref{st2eq} has at most $2b$ solutions: $M_2(a)\le2b$.

If we can find an upper bound for $b$ in terms of $a$, we are ready. To this end, we consider the following expression for the Stirling numbers (see \cite[p. 324]{Charal}, or \cite[(10.2)]{Mezo_stbook}):
\begin{equation}
\sts{n}{n-k}=\sum_{j=0}^k\sts{k+j}{j}_{\ge2}\binom{n}{k+j}\quad(n\ge k+1).\label{st2ast}
\end{equation}
Here $\sts{n}{k}_{\ge2}$ are the associated Stirling numbers, counting partitions of an $n$ element set into $k$ non-empty set, all of cardinality $\ge2$. In particular\footnote{That $\sts{2n}{n}\ge\sts{2n}{n}_{\ge2}$ is obvious by the combinatorial definition of these numbers.},
\[\sts{2b}{b}=\sum_{j=0}^b\sts{b+j}{j}_{\ge2}\binom{2b}{b+j}\ge\sts{2b}{b}_{\ge2}.\]
It is known (and also easy to see), that $\sts{2b}{b}_{\ge2}=\frac{1}{2^b}\frac{(2b)!}{b!}$ (for a more general fact, see p. 286 in \cite{Mezo_stbook}), whence
\[a\ge\sts{2(b-1)}{b-1}\ge\sts{2(b-1)}{b-1}_{\ge2}\ge\left(\frac{b-1}{2}\right)^{b-1}.\]
The definition of the Lambert $W$ function yields that the equation $a=\left(\frac{b-1}{2}\right)^{b-1}$ is solvable in terms of $W$. By monotonicity,
\[b\le1+\frac{\log a}{W\left(\frac12\log a\right)},\]
therefore
\[M_2(a)\le2b\le2+2\frac{\log a}{W\left(\frac12\log a\right)}.\]
This is our first statement. The second one follows by recalling the asymptotics \eqref{Wasymp}.
\end{proof}

\subsection{Bound for $M_1(a)$}

Basically everything can be repeated in the first kind case what we have done in the second kind case. The corresponding version of Lemma 1 holds as well as the first kind version of \eqref{st2ast}. Latter reads as
\begin{equation}
\stf{n}{n-k}=\sum_{j=0}^k\stf{k+j}{j}_{\ge2}\binom{n}{k+j}\quad(n\ge k+1).\label{st1ast}
\end{equation}
Here $\stf{n}{k}_{\ge2}$ is the number of permutations on $n$ letters with $k$ cycles, and having no fixed points.
So
\[\stf{2n}{n}\ge\stf{2n}{n}_{\ge2}.\]
It is clear that $\stf{2n}{n}_{\ge2}=\sts{2n}{n}_{\ge2}$, so our method to estimate $b$ and $M_1(a)$ can be repeated verbatim. We get the following statement.

\begin{theorem}For all $a\ge2$
\[M_1(a)\le2+2\frac{\log a}{W\left(\frac12\log a\right)}.\]
If $a\to\infty$,
\[M_1(a)=O\left(\frac{\log a}{\log\log a-\log\log\log a}\right).\]
\end{theorem}

\section{Numerical tables}

We present all the positive integers $a$ for which $M_i(a)\neq0$, up to $a=100\,000$ ($i=1,2$). The calculations were done by Mathematica v10.1. One can highly facilitate the computations by making use the log-concavity property of the Stirling numbers. This observation shows, that the extremal Stirling numbers (in the second kind case $\sts{n}{2}$ and $\sts{n}{n-1}$) are the smallest ones in a given row (except, of course, the excluded $\sts{n}{1}$ and $\sts{n}{n}$ cases). In addition, we have that $\sts{n}{2}\ge\sts{n}{n-1}=\binom{n}{2}$. Hence, we do not need to check the $n$th row and those rows for which $\binom{n}{2}>a$. We therefore can restrict our search up to the rows where
\[n\le\frac12\left(1+\sqrt{1+8a}\right),\]
and similarly in the first kind case. (This also results the otherwise obvious fact that every number $a$ can occur only finitely many times in a Stirling triangle.)

\vspace{3mm}

\begin{center}
\begin{tabular}{|c|c|c|}
\hline 
Interval&number of $a$ with $M_2(a)>0$ & number of $a$ with $M_1(a)>0$\\
\hline \hline
[0,9999] &176 &169\\ \hline
[10000,19999] &64 &63\\ \hline
[20000,29999] &50 &48\\ \hline
[30000,39999] &42 &41\\ \hline
[40000,49999] &36 &35\\ \hline
[50000,59999] &31 &33\\ \hline
[60000,69999] &32 &31\\ \hline
[70000,79999] &27 &27\\ \hline
[80000,89999] &26 &24\\ \hline
[90000,100000] &24 &25\\ \hline
\end{tabular}
\end{center}

\vspace{3mm}

\begin{center}

Table 1. The number of $a$ with $M_i(a)>0, i=1,2$, where $a$ in the corresponding interval

\end{center}

\section{Further observations and related diophantine equations: results and conjectures}

\subsection{The second kind case}

It might be well true that $M_2(a)$ is a bounded function. We have checked its values up to $a=100\,000$ and we have found that in this range $M_2(a)\le2$ ($a\ge2$). We gave the full table at the end of the paper for those $a$s for which $M_2(a)\neq0$. Here we make some comments on the only interesting cases that we have found, i.e., when $M_2(a)=2$. These are
\[a\in\{15,4\,095,66\,066\}.\]
More precisely,
\[\sts{5}{2}=\sts{6}{5}=15,\]
\[\sts{13}{2}=\sts{91}{90}=4\,095,\]
and
\[\sts{14}{11}=\sts{364}{363}=66\,066.\]
We notice that 15 and $4\,095$ are the so-called Ramanujan-Nagell numbers, solving the diophantine equation
\[2^n-1=\frac{m(m-1)}{2},\]
i.e., the Ramanujan-Nagell numbers are triangular numbers of the form $2^n-1$. There are only five such numbers, $0,1,3,15,4\,095$. Notice, that the above listed Stirling numbers are of the form $\sts{n}{2}$ and $\sts{n}{n-1}$. It is well known that
\[\sts{n}{2}=2^{n-1}-1,\quad\sts{n}{n-1}=\binom{n}{2},\]
so the appearance of the Ramanujan-Nagell numbers is not a surprise.

What about $a=66\,066$? It arises in $\sts{n}{n-3}=\sts{m}{m-1}$ with
\[(n,m)=(14,364).\]
The $\sts{n}{n-3}$ special Stirling numbers have a closed form \cite[p. 392]{Mezo_stbook}:
\[\sts{n}{n-3}=\binom{n}{4}+10\binom{n}{5}+15\binom{n}{6}.\]
(This is also a consequence of \eqref{st2ast}). One might therefore ask, whether there are other pairs $(n,m)$ such that
\begin{equation}
\binom{n}{4}+10\binom{n}{5}+15\binom{n}{6}=\binom{m}{2}.\label{diof1}
\end{equation}

Using a simple MAPLE program we find that the unique solution of \eqref{diof1} with $6\leq n\leq 4.6\cdot 10^7$ is $(n,m)=(14,364)$. Based on our numerical investigations we give the following

\begin{conjec}
The unique solution of \eqref{diof1} with $n\geq 6 $  is $(n,m)=(14,364)$. 
\end{conjec}

Unfortunately, the solution to the above conjecture is out of reach of current methods. 

\subsection{The first kind case}

In the first kind Stirling triangle there is no abundance of values appearing more than once. In fact, up to $a=100000$, there are only three numbers:
\[a\in\{1,6,120\}.\]
All of them appear twice (so $M_1(a)\le2$ for $1<a\le100\,000$). They are factorial numbers and triangular numbers at the same time: 1 appears infinitely many often on the extremes, and the other two numbers appear at the following positions:
\[\stf{4}{1}=\stf{4}{3}=6,\]
\[\stf{6}{1}=\stf{16}{15}=120.\]
(Remember, that $\stf{n}{1}=(n-1)!$ while $\stf{n}{n-1}=\binom{n}{2}$.)

The question naturally arises: are there other factorial-triangular numbers? In other words, are there solutions to the diophantine equation
\begin{equation}
n!=\frac{m(m-1)}{2}\label{diof2}
\end{equation}
other than $(n,m)=\{(1,2),(3,4),(5,16)\}$?

Equation \eqref{diof2}, or equivalently 
\begin{equation}
8n!+1=(2m-1)^2 \label{equiv}
\end{equation}
 is a member of the family of polynomial-factorial equations. The most famous problem to this direction is the unsolved Brocard-Ramanujan equation 
$$n!+1=m^2,$$ 
with three known solutions $(m,n)=(4,5), (5,11)$ and $(7,71)$. Before using the approach of \cite{Overholt} and \cite{berndt} to obtain our theorem we recall the famous $abc$-conjecture that is one of the most exciting open problems in diophantine number theory. 

\begin*{\bf ABC-Conjecture.}
{\it  For any $\epsilon > 0$ there exists a constant $C(\epsilon)$
depending only on $\epsilon$ such that whenever $A, B$ and $C$ are three coprime and
non-zero integers with $A + B = C$, then
$$\max(|A|, |B|, |C|) < C(\epsilon)Rad(ABC)^{1+\epsilon},$$
where $Rad(ABC)$ denotes the radical of $ABC$, i. e. $Rad(ABC)=\prod_{p|ABC}p$.}
\end*{}

Let $k\geq 3$ be an integer and denote by 
$$P_k(x)=\frac{x((k-2)x-k+4)}{2}$$
the $x$th $k$-gonal number. For $k=3$ and $4$, we have $P_k(x)=\binom{x+1}{2}$ and $x^2$, respectively. For fixed $k\neq 4$, we consider the equation
$n!=P_k(x)$, or equivalently
\begin{equation}
8(k-2)n!+(k-4)^2=(2(k-2)x+(4-k))^2 \label{main}
\end{equation}
in positive integers $n>1$ and $x>1$.

\begin{theorem} For fixed $k\neq 4$, the ABC-Conjecture implies \eqref{main} has only finitely many in positive integers  $n>1$ and $y>1$. Further, all the solutions $(k,n,x)$ with $3\leq k\leq 50, 1<n\leq 10^5$  and $1<x$ are $(k,n,x)=(3,3,3),(3,5,15),(6,3,2),(6,5,8),(9,4,3), (24,4,2)$ and $(41,5,3)$.

\end{theorem}

If $k=4$ we have $n!=x^2$, and Bertrand-postulate gives that it has a unique solution $(n,x)=(1,1 )$.

\begin{proof}

We may suppose that $n\geq k-2$, and by Erd\H{o}s Lemma on the product of all primes less than or equal to $n$ we get 
$$ Rad(8(k-2)\cdot n!\cdot (k-4)^2\cdot (2(k-2)x+(4-k))^2)\leq 4^n\cdot Rad(2(k-2)x+(4-k))\leq 4^n\cdot\left( 2(k-2)x+(4-k)\right),$$
and from \eqref{main}
$$2(k-2)x+(4-k))\leq (8n\cdot n!+(n-2)^2)^{\frac{1}{2}}.$$
Assuming $ABC$-conjecture these inequalities yield 
$$
n!<C(\epsilon)\cdot  (8n\cdot n!+(n-2)^2)^{\frac{1}{2}+\epsilon}, 
$$
which is impossible for large $n$.

For the proof of the second part of our theorem we follow the local argument of \cite{berndt}. Let $P$ be the set of first $19$ primes over $10^5$, that is $P=\{100019,100043,100049,\ldots ,100267,100271\}$. If we can give a prime $p\in P$ such that the Legendre symbol 
$$\left(\frac{8(k-2)n!+(k-4)^2}{p}\right)=-1$$
then \eqref{main} is impossible in rational integers. A short MAPLE program was implemented for a usual PC, the running time is about 30-50 minutes for every $k$. This sieve is pretty efficient, the survivor pairs $(k,n)$ under the conditions of the theorem are $(k,n)=(5,2); (k,3), k\in\{3,6,9,14,23,42\}; (k,4), k\in\{9,24,27\}; (k,5),k\in \{3,6,8,30,32,41\}; (41,8)$ and $(5, 54545), (12,93137), (17,12797), (28,78842)$,
$ (33,53361), (35,92666)$, and $(38,11846)$.  For $n\leq 8$, a direct calculation gives the solutions $(k,n,x)= (3,3,3), (3,5,15), (6,3,2), 
(6,5,8), (9,4,3), (24,4,2)$ and $(41,5,3)$. One can check by MAPLE that
$$\left(\frac{8\cdot 3\cdot 54545!+1}{100279}\right)=-1,
\left(\frac{8\cdot 10 \cdot 93137!+64}{100279}\right)=-1,
\left(\frac{8\cdot 15\cdot 12797!+169}{100279}\right)=-1,$$
$$\left(\frac{8\cdot 26\cdot 78842!+24^2}{100297}\right)=-1,
\left(\frac{8\cdot 31\cdot53361!+29^2}{100279}\right)=-1,
\left(\frac{8\cdot 33\cdot 92666!+31^2}{100313}\right)=-1,$$
$$\left(\frac{8\cdot 36\cdot 11846!+34^2}{100291}\right)=-1,$$
thus Theorem 3 is proved.
\end{proof}
Considering these results one can state
\begin{conjec}
All the solutions of \eqref{diof2} are $(n,m)=(1,2),(3,4),$ and $(5,16)$.
\end{conjec}

\section{Appendix}

Here we give the full table for those values of $a$s for which $M_i(a)\neq0, i=1,2$. The duplicated values are not deleted in the below tables.

\vspace{5mm}
The second kind case.

$\mathbf{a\in[0,9\,999]}$:

0, 1, 3, 6, 7, 10, 15, 15, 21, 25, 28, 31, 36, 45, 55, 63, 65, 66, 78, 90, 91, 105, 120, 127, 136, 140, 153, 171, 190, 210, 231, 253, 255, 266, 276, 300, 301, 325, 350, 351, 378, 406, 435, 462, 465, 496, 
511, 528, 561, 595, 630, 666, 703, 741, 750, 780, 820, 861, 903, 946, 
966, 990,

1023, 1035, 1050, 1081, 1128, 1155, 1176, 1225, 1275, 1326, 1378, 1431, 1485, 1540, 1596, 1653, 1701, 1705, 1711, 1770, 1830, 1891, 1953,

2016, 2047, 2080, 2145, 2211, 2278, 2346, 2415, 2431, 2485, 2556, 2628, 2646, 2701, 2775, 2850, 2926,

3003, 3025, 3081, 3160, 3240, 3321, 3367, 3403, 3486, 3570, 3655, 3741, 3828, 3916,

4005, 4095, 4095, 4186, 4278, 4371, 4465, 4550, 4560, 4656, 4753, 4851, 4950,

5050, 5151, 5253, 5356, 5460, 5565, 5671, 5778, 5880, 5886, 5995,

6020, 6105, 6216, 6328, 6441, 6555, 6670, 6786, 6903, 6951,

7021, 7140, 7260, 7381, 7503, 7626, 7750, 7770, 7820, 7875,

8001, 8128, 8191, 8256, 8385, 8515, 8646, 8778, 8911,

9045, 9180, 9316, 9330, 9453, 9591, 9730, 9870, 9996

\vspace{5mm}
$\mathbf{a\in[10\,000,19\,999]}$:

10011, 10153, 10296, 10440, 10585, 10731, 10878, 11026, 11175, 11325, \
11476, 11628, 11781, 11880, 11935, 12090, 12246, 12403, 12561, 12597, \
12720, 12880, 13041, 13203, 13366, 13530, 13695, 13861, 14028, 14196, \
14365, 14535, 14706, 14878, 15051, 15225, 15400, 15576, 15675, 15753, \
15931, 16110, 16290, 16383, 16471, 16653, 16836, 17020, 17205, 17391, \
17578, 17766, 17955, 18145, 18336, 18528, 18721, 18915, 19110, 19285, \
19306, 19503, 19701, 19900

\vspace{5mm}
$\mathbf{a\in[20\,000,29\,999]}$:

20100, 20301, 20503, 20706, 20910, 21115, 21321, 21528, 21736, 21945, \
22155, 22275, 22366, 22578, 22791, 22827, 23005, 23220, 23436, 23485, \
23653, 23871, 24090, 24310, 24531, 24753, 24976, 25200, 25425, 25651, \
25878, 26106, 26335, 26565, 26796, 27028, 27261, 27495, 27730, 27966, \
28203, 28336, 28441, 28501, 28680, 28920, 29161, 29403, 29646, 29890

\vspace{5mm}
$\mathbf{a\in[30\,000,39\,999]}$:

30135, 30381, 30628, 30876, 31125, 31375, 31626, 31878, 32131, 32385, \
32640, 32767, 32896, 33153, 33411, 33670, 33902, 33930, 34105, 34191, \
34453, 34716, 34980, 35245, 35511, 35778, 36046, 36315, 36585, 36856, \
37128, 37401, 37675, 37950, 38226, 38503, 38781, 39060, 39325, 39340, \
39621, 39903

\vspace{5mm}
$\mathbf{a\in[40\,000,49\,999]}$:

40186, 40250, 40470, 40755, 41041, 41328, 41616, 41905, 42195, 42486, \
42525, 42778, 43071, 43365, 43660, 43956, 44253, 44551, 44850, 45150, \
45451, 45753, 46056, 46360, 46665, 46971, 47278, 47450, 47586, 47895, \
48205, 48516, 48828, 49141, 49455, 49770

\vspace{5mm}
$\mathbf{a\in[50\,000,59\,999]}$:

50086, 50403, 50721, 51040, 51360, 51681, 52003, 52326, 52650, 52975, \
53301, 53628, 53956, 54285, 54615, 54946, 55278, 55575, 55611, 55945, \
56280, 56616, 56953, 57291, 57630, 57970, 58311, 58653, 58996, 59340, \
59685

\vspace{5mm}
$\mathbf{a\in[60\,000,69\,999]}$:

0031, 60378, 60726, 61075, 61425, 61776, 62128, 62481, 62835, 63190, \
63546, 63903, 63987, 64261, 64620, 64701, 64980, 65341, 65535, 65703, \
66066, 66066, 66430, 66795, 67161, 67528, 67896, 68265, 68635, 69006, \
69378, 69751

\vspace{5mm}
$\mathbf{a\in[70\,000,79\,999]}$:
70125, 70500, 70876, 71253, 71631, 72010, 72390, 72771, 73153, 73536, \
73920, 74305, 74691, 74907, 75078, 75466, 75855, 76245, 76636, 77028, \
77421, 77815, 78210, 78606, 79003, 79401, 79800

\vspace{5mm}
$\mathbf{a\in[80\,000,89\,999]}$:

80200, 80601, 81003, 81406, 81810, 82215, 82621, 83028, 83436, 83845, \
84255, 84666, 85078, 85491, 85905, 86275, 86320, 86526, 86736, 87153, \
87571, 87990, 88410, 88831, 89253, 89676

\vspace{5mm}
$\mathbf{a\in[90\,000,100\,000]}$:

90100, 90525, 90951, 91378, 91806, 92235, 92665, 93096, 93528, 93961, \
94395, 94830, 95266, 95703, 96141, 96580, 97020, 97461, 97903, 98346, \
98790, 98890, 99235, 99681

\vspace{5mm}
The first kind case.

$\mathbf{a\in[0,9\,999]}$:

0, 1, 2, 3, 6, 6, 10, 11, 15, 21, 24, 28, 35, 36, 45, 50, 55, 66, 78, 85, 91, \
105, 120, 120, 136, 153, 171, 175, 190, 210, 225, 231, 253, 274, 276, \
300, 322, 325, 351, 378, 406, 435, 465, 496, 528, 546, 561, 595, 630, \
666, 703, 720, 735, 741, 780, 820, 861, 870, 903, 946, 990,

1035, \
1081, 1128, 1176, 1225, 1275, 1320, 1326, 1378, 1431, 1485, 1540, \
1596, 1624, 1653, 1711, 1764, 1770, 1830, 1891, 1925, 1953, 1960,

2016, 2080, 2145, 2211, 2278, 2346, 2415, 2485, 2556, 2628, 2701, \
2717, 2775, 2850, 2926,

3003, 3081, 3160, 3240, 3321, 3403, 3486, 3570, 3655, 3731, 3741, 3828, 3916,

4005, 4095, 4186, 4278, 4371, 4465, 4536, 4560, 4656, 4753, 4851, 4950,

5005, 5040, 5050, 5151, 5253, 5356, 5460, 5565, 5671, 5778, 5886, 5995,

6105, 6216, 6328, 6441, 6555, 6580, 6670, 6769, 6786, 6903,

7021, 7140, 7260, 7381, 7503, 7626, 7750, 7875,

8001, 8128, 8256, 8385, 8500, 8515, 8646, 8778, 8911,

9045, 9180, 9316, 9450, 9453, 9591, 9730, 9870

\vspace{5mm}
$\mathbf{a\in[10\,000,19\,999]}$:
10011, 10153, 10296, 10440, 10585, 10731, 10812, 10878, 11026, 11175, \
11325, 11476, 11628, 11781, 11935, 12090, 12246, 12403, 12561, 12720, \
12880, 13041, 13068, 13132, 13203, 13366, 13530, 13566, 13695, 13861, \
14028, 14196, 14365, 14535, 14706, 14878, 15051, 15225, 15400, 15576, \
15753, 15931, 16110, 16290, 16471, 16653, 16815, 16836, 17020, 17205, \
17391, 17578, 17766, 17955, 18145, 18150, 18336, 18528, 18721, 18915, \
19110, 19306, 19503, 19701, 19900

\vspace{5mm}
$\mathbf{a\in[20\,000,29\,999]}$:

20100, 20301, 20503, 20615, 20706, 20910, 21115, 21321, 21528, 21736, \
21945, 22155, 22366, 22449, 22578, 22791, 23005, 23220, 23436, 23653, \
23871, 24090, 24310, 24531, 24753, 24976, 25025, 25200, 25425, 25651, \
25878, 26106, 26335, 26565, 26796, 27028, 27261, 27495, 27730, 27966, \
28203, 28441, 28680, 28920, 29161, 29403, 29646, 29890

\vspace{5mm}
$\mathbf{a\in[30\,000,39\,999]}$:

30107, 30135, 30381, 30628, 30876, 31125, 31375, 31626, 31878, 32131, \
32385, 32640, 32670, 32896, 33153, 33411, 33670, 33930, 34191, 34453, \
34716, 34980, 35245, 35511, 35778, 35926, 36046, 36315, 36585, 36856, \
37128, 37401, 37675, 37950, 38226, 38503, 38781, 39060, 39340, 39621, \
39903

\vspace{5mm}
$\mathbf{a\in[40\,000,49\,999]}$:

40186, 40320, 40470, 40755, 41041, 41328, 41616, 41905, 42195, 42486, \
42550, 42778, 43071, 43365, 43660, 43956, 44253, 44551, 44850, 45150, \
45451, 45753, 46056, 46360, 46665, 46971, 47278, 47586, 47895, 48205, \
48516, 48828, 49141, 49455, 49770

\vspace{5mm}
$\mathbf{a\in[50\,000,59\,999]}$:

50050, 50086, 50403, 50721, 51040, 51360, 51681, 52003, 52326, 52650, \
52975, 53301, 53628, 53956, 54285, 54615, 54946, 55278, 55611, 55770, \
55945, 56280, 56616, 56953, 57291, 57630, 57970, 58311, 58500, 58653, \
58996, 59340, 59685

\vspace{5mm}
$\mathbf{a\in[60\,000,69\,999]}$:

60031, 60378, 60726, 61075, 61425, 61776, 62128, 62481, 62835, 63190, \
63273, 63546, 63903, 64261, 64620, 64980, 65341, 65703, 66066, 66430, \
66795, 67161, 67284, 67528, 67896, 67977, 68265, 68635, 69006, 69378, \
69751

\vspace{5mm}
$\mathbf{a\in[70\,000,79\,999]}$:

70125, 70500, 70876, 71253, 71631, 72010, 72390, 72771, 73153, 73536, \
73920, 74305, 74691, 75078, 75466, 75855, 76245, 76636, 77028, 77421, \
77815, 78210, 78561, 78606, 79003, 79401, 79800

\vspace{5mm}
$\mathbf{a\in[80\,000,89\,999]}$:

80200, 80601, 81003, 81406, 81810, 82215, 82621, 83028, 83436, 83845, \
84255, 84666, 85078, 85491, 85905, 86320, 86736, 87153, 87571, 87990, \
88410, 88831, 89253, 89676

\vspace{5mm}
$\mathbf{a\in[90\,000,100\,000]}$:

90100, 90335, 90525, 90951, 91091, 91378, 91806, 92235, 92665, 93096, \
93528, 93961, 94395, 94830, 95266, 95703, 96141, 96580, 97020, 97461, \
97903, 98346, 98790, 99235, 99681

%





\section*{Acknowledgements}

The research of A.B. was supported in part by the HUN-REN Hungarian Research Network and by the NKFIH grants ANN130909 and K128088 of the Hungarian National Research, Development and Innovation Office.


\begin{thebibliography}{99}

\bibitem{AEH}
H. L. Abbott, P. Erd\H{o}s, D. Hanson, On the number of times an integer occurs as a binomial coefficient, 81(3) (1974), 256-261.

\bibitem{berndt}
B. C. Berndt and W. F. Galway, On the Brocard-Ramanujan Diophantine equation $n!+1=m^2$, Ramanujan Journal 4(1) (2000), 41-42.

\bibitem{Charal}
Ch. A. Charalambides, Enumerative Combinatorics, Chapman \& Hall / CRC, 2002.

\bibitem{W}
R. M. Corless, G. H. Gonnet, D. E. G. Hare, D. J. Jeffrey, D. E. Knuth, On the Lambert $W$ function, Adv. Comput. Math. 5 (1996), 329-359.

\bibitem{FPP} J. Ferenczik, \'A. Pint{\'e}r, B. Porv{\'a}zsnyik, On equal values of {Stirling} numbers of the second kind, Appl. Math. Comput., 218(3) (2011), 980-984.

\bibitem{Kane1}
D. M. Kane, On the number of representations of $t$ as a binomial coefficient, Integers 4 (2004), \#A07.

\bibitem{Kane2}
D. M. Kane, Improved bounds on the number of ways of expressing $t$ as a binomial coefficient, Integers 7 (2007), \#A53.

\bibitem{Konincketal}
J.-M. de Koninck, N. Doyon, W. Verreault, Repetitions of multinomial coefficients and a generalization of Singmaster's conjecture, Integers 21 (2021), \#A34.


\bibitem{Taoetal}
K. Matom\"aki, M. Radziwi\l\l, X. Shao, T. Tao, J. Ter\"av\"ainen, Singmaster’s conjecture in the interior of Pascal’s triangle, Quart. J. Math. 73(3) 2022, 1137-1177.

\bibitem{Mezo}
I. Mez\H{o}, The Lambert $W$ Function: Its Generalizations and Applications, Chapman and Hall/CRC, 2022.

\bibitem{Mezo_stbook}
I. Mez\H{o}, Combinatorics and Number Theory of Counting Sequences, Chapman and Hall/CRC, 2019.

\bibitem{AMM}
D. A. Lind, The quadratic field $\mathbb Q(\sqrt 5)$ and a certain diophantine equation, Fibonacci Q. 6(3) 1968, 86-93.

\bibitem{Overholt}
M. Overholt, The diophantine equation $n^2+1=m^2$, Bulletin London. Math. Soc. 25(2) (1993), 104.

\bibitem{Singm1}
D. Singmaster, Research Problems: How often does an integer occur as a binomial coefficient?, Amer. Math. Monthly 78(4) (1971), 385-386.

\bibitem{Singm2}
D. Singmaster,  Repeated binomial coefficients and Fibonacci numbers, Fibonacci Quart. 13(4) (1975), 295-298.

\bibitem{deweger}
B. M. M. de Weger, Equal Binomial Coefficients: Some Elementary Considerations, J. Number Theory 63 (1997), 373-386. 

\end{thebibliography}
\end{document}